\documentclass{amsart}
\usepackage{amsmath,amsthm,amscd,amssymb,amsfonts, amsbsy}
\usepackage{latexsym}
\usepackage{txfonts}

\numberwithin{equation}{section}

\theoremstyle{plain}
\newtheorem{theorem}[equation]{Theorem}
\newtheorem{lemma}[equation]{Lemma}

\theoremstyle{definition}

\theoremstyle{remark}

\newcommand{\bR}{\mathbb R}

\newcommand{\bZ}{\mathbb Z}
\newcommand{\bE}{\mathbb E}

\newcommand\cC{\mathcal{C}}

\newcommand\U{\mathcal{U}}

\newcommand{\set}[1]{\left\{#1\right\}}
\newcommand{\norm}[1]{\lVert#1\rVert}

\newcommand{\abs}[1]{\left\lvert#1\right\rvert}

\renewcommand{\epsilon}{\varepsilon}

\newcommand{\Co}{\mathcal{C}^{1,2}}

\newcommand{\Ho}{\mathcal{C}^{1+\alpha/2,2+\alpha}}

\begin{document}
\title[On a degenerate parabolic equation]{On a degenerate parabolic equation arising in pricing of Asian options}

\author{Seick Kim}
\address{Department of Mathematics, Yonsei University, 262 Seongsanno, Seodaemun-gu, Seoul 120-749, Korea}
\email{kimseick@yonsei.ac.kr}

\begin{abstract}
We study a certain one dimensional, degenerate parabolic partial differential
equation with a boundary condition which arises in pricing of Asian options. 
Due to degeneracy of the partial differential operator and the non-smooth
boundary condition, regularity of the generalized solution of such a
problem remained unclear.
We prove that the generalized solution of the problem is indeed a classical
solution.
\end{abstract}

\maketitle

\section{Introduction and Main result}\label{sec:intro}
In \cite{Vecer02}, Ve\v{c}e\v{r} proposed a unified method for pricing Asian options, which lead to a simple one-dimensional partial differential equation
\begin{equation}
\label{eq:I03c}
u_t+\tfrac{1}{2}\left(x-e^{-\int_0^t d\nu(s)}q(t) \right)^2 \sigma^2 u_{xx}=0
\end{equation}
with the boundary condition
\begin{equation}
\label{eq:I04d}
u(T,x)=(x-K_1)_{+}:= \max(x-K_1,0).
\end{equation}
Here, $\nu(t)$ is the measure representing the dividend yield, $\sigma$
is the volatility of the underlying asset, $q(t)$ is the trading strategy
given by
\begin{equation*}
q(t)=\exp\left\{-\int_t^T \,d\nu(s)\right\}\cdot\int_t^T
\exp\left\{-r(T-s)+\int_s^T \,d\nu(\tau)\right\}\,d\mu(s),
\end{equation*}
where $r$ is the interest rate and $\mu(t)$ represents a general weighting
factor.
In the fixed strike Asian call option, we have $K_1=0$ in the boundary
condition \eqref{eq:I04d}; see \cite{Vecer01, Vecer02} for details.
If we assume that $d\mu(t)=\rho(t)\,dt$ for some $\rho\in L^\infty([0,T])$ 
satisfying $0<\rho_0\leq\rho(t)$, then it is readily seen that
\begin{equation*}
e^{-\int_0^t d\nu(s)}q(t)=c\int_t^T
\exp\left\{-r(T-s)+\int_s^T \,d\nu(\tau)\right\}\,d\mu(s)
\quad \left(c=e^{-\int_0^T\,d\nu(s)} >0\right)
\end{equation*}
is a monotone decreasing Lipschitz continuous function.
We are thus lead to consider the following one-dimensional parabolic PDE
\begin{equation}
\label{eq:I05e}
u_t+\tfrac{1}{2}(b(t)-x)^2\,u_{xx}=0
\end{equation}
in $H_T:=(0,T)\times \bR$ with the boundary condition
\begin{equation}
\label{eq:I06a}
u(T,x)=x_{+},
\end{equation}
where $b(t)$ is a Lipschitz continuous function defined on $[0,T]$
such that $b(T)=0$ and 
\begin{equation}
\label{eq:I07b}
m_1\le -b'(t) \le m_2,\,\text{ for a.e. }\,t\in(0,T)\,\text{ for some }\,
m_1, m_2>0.
\end{equation}

In this article we are mainly concerned with regularity of the (generalized)
solution $u(t,x)$ of the problem \eqref{eq:I05e}, \eqref{eq:I06a}.
It is a rather nontrivial task to show that the problem \eqref{eq:I05e},
\eqref{eq:I06a} has a solution in the classical sense.
First of all, it should be noted that even though the coefficient which appears
in \eqref{eq:I05e} is Lipschitz continuous, the classical approach based on 
Schauder theory is not applicable here, for the operator in \eqref{eq:I05e}
becomes degenerate along the curve $x=b(t)$.
Nevertheless, it is possible to show that the problem \eqref{eq:I05e},
\eqref{eq:I06a} admits the ``probabilistic'' solution: Let
\begin{equation}
\label{eq:I09u}
u(t,x):=\bE f(X_{T}(t,x)),
\end{equation}
where $f(x):=x_{+}$ and $X_s$ is the stochastic process which satisfies,
for $t\in[0,T]$ and $x\in\bR$,
\begin{equation}
\label{eq:I08z}
\left\{\begin{array}{l}
dX_s(t,x)=(b_s-X_s(t,x))\,dw_s,\quad s\ge t, \quad (\,b_s=b(s)\,)\\
X_{t}(t,x)=x.
\end{array}\right.
\end{equation}
It is known that such a process $X_t$ exists and that if $f$ is twice
continuously differentiable, then $u(t,x)$ given by \eqref{eq:I09u} is
a classical solution of
\eqref{eq:I05e} in $H_T$ (i.e., $u(t,x)$ is continuously
differentiable with respect $t$ and twice continuously differentiable with
respect to $x$ in $H_T$ and satisfies \eqref{eq:I05e} there)
with the boundary condition $u(T,x)=f(x)$;
see e.g. \cite{Krylov}.
Unfortunately, $f(x)=x_{+}$ is not twice continuously differentiable and
the above method is not directly applicable here.
On the other hand, it should be also noted that if $b(t)$ is smooth enough and
$b'(t)\neq 0$ everywhere, then the differential operator in \eqref{eq:I05e}
satisfies H\"ormander's conditions for hypoellipticity (see \cite{Hormander}).
Therefore, in this case, it is not hard to see that $u(t,x)$ given by
\eqref{eq:I09u} becomes a classical solution of the problem \eqref{eq:I05e},
\eqref{eq:I06a}.
However, H\"ormander's theorem is not available under a mere assumption that
$b(t)$ is a Lipschitz continuous function satisfying \eqref{eq:I07b}.

The main goal of this article is to present a technique to prove that
the generalized solution $u(t,x)$ of the problem \eqref{eq:I05e},
\eqref{eq:I06a} is indeed a classical solution.
Let us now state our main theorem.
\begin{theorem}
\label{thm:I01}
For $t\in[0,T]$ and $x\in\bR$, let $X_s=X_s(t,x)$ be the stochastic process
which satisfies \eqref{eq:I08z} and let $u(t,x)$ be defined as in
\eqref{eq:I09u} with $f(x):=x_{+}$.
Then $u(t,x)$ is a classical solution of the equation \eqref{eq:I05e} in
$H_T=(0,T)\times \bR$ satisfying the boundary condition \eqref{eq:I06a}.
\end{theorem}

The organization of this paper is as follows. In Sec.~\ref{sec:notation},
we introduce some notations and present a preliminary lemma which will be
used in the proof of the main result.
In Sec.~\ref{sec:proof}, we give the proof of our main result,
Theorem~\ref{thm:I01}.
An outline of the proof is as follows.
We first split $u=u_1+u_2$, where $u_i$ are the probabilistic
solutions of \eqref{eq:I05e} satisfying $u_i(T,x)=f_i(x)$
with $f_1(x)=x$ and $f_2(x)=(-x)_{+}$.
It can be readily seen that $u_1$ is a classical solution of \eqref{eq:I05e}
in $H_T$.
Next, we show that $u_2\equiv 0$ in the set
$\set{(t,x)\in H_T: x\geq b(t)}$.
Then, by using a suitable rescaling and the lemma in Sec.~\ref{sec:notation},
we show that $u_2$ decays very rapidly to zero near the curve $x=b(t)$.
This is the key point of the proof.
Then, we apply the interior Schauder estimates to $u_2$ to conclude that
$\partial_t u_2$, $\partial_x u_2$, and $\partial_{xx} u_2$
all decay rapidly to zero near the curve $x=b(t)$, from which we will be able
to complete the proof.
Finally, In Sec.~\ref{sec:extra}, we reformulate the key lemma of the proof
in more general settings, in the hope that this technique might be useful
to some other problems as well.

\section{Notations and preliminaries}\label{sec:notation}
\subsection{Some notations}\label{sec:N01}
We introduce some notations which will be used in the proof.
We define the parabolic distance between the points $z_1=(t_1,x_1)$ and
$z_2=(t_2,x_2)$ as
\begin{equation*}
|z_1-z_2|_p:=\max(\sqrt{|t_1-t_2|},|x_1-x_2|).
\end{equation*}
Let $\alpha\in(0,1)$ be a fixed constant.
If $u$ is a function in a domain $Q\subset \bR^2$, we denote
\begin{align*}
[u]_{\alpha/2,\alpha;Q}&=\sup_{\substack{z_1\neq z_2\\ z_1,z_2\in Q}}
\frac{|u(z_1)-u(z_2)|}{|z_1-z_2|_p^\alpha},\quad
|u|_{0;Q}=\sup_Q |u|,\\
|u|_{\alpha/2,\alpha;Q}&=|u|_{0;Q}+[u]_{\alpha/2,\alpha;Q}.
\end{align*}
By $\cC^{\alpha/2,\alpha}(Q)$ we denote the space of all functions for which
$|u|_{\alpha/2,\alpha;Q}<\infty$.
We also introduce the space $\Ho(Q)$ as the set of all functions $u$ defined
in $Q$ for which both
\begin{align*}
[u]_{1+\alpha/2,2+\alpha;Q}&:=[u_t]_{\alpha/2,\alpha;Q}+
[u_{xx}]_{\alpha/2,\alpha;Q}<\infty\quad\text{and}\\
|u|_{1+\alpha/2,2+\alpha;Q}&:=|u|_{0;Q}+|u_x|_{0;Q}+|u_t|_{0,Q}+|u_{xx}|_{0,Q}+
[u]_{1+\alpha/2,2+\alpha;Q}<\infty.
\end{align*}
The function space $\Co(Q)$ denotes the set of all functions defined
in $Q$ for which
\begin{equation*}
|u|_{0;Q}+|u_x|_{0;Q}+|u_t|_{0,Q}+|u_{xx}|_{0,Q}<\infty.
\end{equation*}
We say $u\in \Ho_{loc}(Q)$ if $u \in \Ho(Q')$ for all compact set $Q'\Subset Q$
and similarly, $u\in \Co_{loc}(Q)$ if $u \in \Co(Q')$ for all compact set
$Q'\Subset Q$.
\subsection{A lemma on Gaussian estimates}\label{sec:G01}
Let $R>0$ be fixed and $g(x)$ be a continuous function defined on $[-R,R]$
satisfying $1/2\leq g(x)\leq 3/2$ for $x\in [-R,R]$.
We denote
\begin{align*}
Q&:=\{(t,x)\in \bR^2: 0<t<2, |x|<R\},\\
\Omega&:=\{(t,x)\in Q: t>g(x)\},\quad \Sigma:=\{(t,x)\in Q: t=g(x) \}.
\end{align*}
\begin{lemma}
\label{lem:G01}
Let $\Omega$ and $\Sigma$ be defined as above and let $a(t,x)$ be a function
satisfying
\begin{equation}
\label{eq:G01q}
0\leq a(t,x)\leq 1,\quad\forall (t,x)\in \Omega.
\end{equation}
Assume that $u\in \Co_{loc}(\Omega)\cap \cC(\overline \Omega)$ and satisfies
\begin{equation*}
\left\{\begin{array}{l}
Lu:=u_t-a(t,x) u_{xx}=0\quad\text{in}\quad\Omega\\
u=0\quad\text{on}\quad\Sigma.
\end{array}\right.
\end{equation*}
Then, we have the following estimate:
\begin{equation}
\label{eq:G01y}
|u|_{0;\Omega'} \leq
(16/\sqrt{2\pi})\,R^{-1}e^{-R^2/32}\,|u|_{0;\Omega},\quad
\text{where }\,\Omega':=\{(t,x)\in\Omega: |x|< R/2\}.
\end{equation}
\end{lemma}
\begin{proof}
By changing $u\to u/\abs{u}_{0;\Omega}$, we may assume $\abs{u}_{0;\Omega}=1$.
Let $\Phi(t,x)$ be the fundamental solution of the heat equation
in $(0,\infty)\times \bR$; i.e.,
\begin{equation*}
\Phi(t,x)=\frac{1}{\sqrt{4\pi t}}\,e^{-x^2/4t}.
\end{equation*}
let $v(t,x)$ be a function on $(0,\infty)\times \bR$ defined by
\begin{equation}
\label{eq:G05i}
v(t,x)=2 \int_{E} \Phi(t,x-y)\,dy,\quad\text{where }\,
E:=\bigcup_{j\in\bZ}\,((4j+1)R,(4j+3)R).
\end{equation}
Denote $D=\{(t,x)\in \bR^2: t>0,\,|x|<R\}$.
From \eqref{eq:G05i}, it follows that $v\geq 0$ and satisfies
\begin{equation}
\label{eq:G05j}
\left\{\begin{array}{l}
v_t-v_{xx}=0\quad\text{in}\quad D,\\
v=0\quad\text{on}\quad\partial_t D:= \{(t,x)\in \bR^2: t=0,\,|x|<R\},\\
v=1\quad\text{on}\quad\partial_x D:= \{(t,x)\in \bR^2: t>0,\,|x|=R\}.
\end{array}\right.
\end{equation}
Moreover, by the comparison principle, we see that
$v(t,x)\le v(t+h,x)$ in $D$ for any $h>0$, and thus it follows that
\begin{equation}
\label{eq:G06a}
v_{xx}=v_t\ge 0\quad\text{in}\quad D.
\end{equation}
Then by using \eqref{eq:G01q}, we have
\begin{equation*}
L(v\pm u)=Lv=v_t-a(t,x)v_{xx}\geq v_t-v_{xx}=0\quad\text{in}\quad \Omega.
\end{equation*}
Denote by $\partial_p\Omega$ the parabolic boundary of $\Omega$
(see e.g., \cite{Lieberman} for its definition) and observe that
$\Sigma':=\partial_p\Omega\setminus \Sigma \subset \partial_x D$.
Then, by \eqref{eq:G05j}, we find (recall that we assume $|u|_{0;\Omega}=1$)
\begin{equation*}
v \pm u \geq 0 \quad\text{on}\quad \partial_p\Omega.
\end{equation*}
Therefore, by the maximum principle and \eqref{eq:G06a}, we have
\begin{equation*}
|u(t,x)|\leq v(t,x) \leq v(2,x),\quad\forall (t,x)\in \Omega.
\end{equation*}
On the other hand, for $|x|<R/2$, we estimate $v(2,x)$ by
\begin{align}
\label{eq:G07b}
v(2,x)&=2\int_E \Phi(2,x-y)\,dy \leq 4\int_{R-|x|}^\infty\Phi(2,y)\,dy
\leq 4\int_{R/2}^\infty\Phi(2,y)\,dy\\
\nonumber
&\leq \frac{8}{\sqrt{8\pi}R}\int_{R/2}^\infty ye^{-y^2/8}\,dy=
\frac{16}{\sqrt{2\pi}}\,R^{-1}e^{-R^2/32}.
\end{align}
The lemma is proved.
\end{proof}

\section{Proof of Theorem~\ref{thm:I01}}\label{sec:proof}
For $t\in[0,T]$ and $x\in\bR$, let $X_s=X_s(t,x)$ be the stochastic process
which satisfies \eqref{eq:I08z}.
It is well known that such a process $X_t$ exists; see e.g.,
\cite[Theorem V.1.1]{Krylov}.
Denote
\begin{equation}
\label{eq:P01a}
u_1(t,x)= \bE f_1(X_T(t,x)),\quad
u_2(t,x)= \bE f_2(X_T(t,x)),\quad
\end{equation}
where $f_1(x)=x$ and $f_2(x)=(-x)_{+}$ so that $f(x)=f_1(x)+f_2(x)$.
By \cite[Theorem V.7.4]{Krylov}, the function $u_1$ and its derivatives
$\partial_t u_1$, $\partial_x u_1$, and $\partial_{xx} u_1$  are continuous in
$H_T$ and $u_1$ satisfies the equation \eqref{eq:I05e} there.
In other words, the function $u_1$ is a classical solution of \eqref{eq:I05e}
in $H_T$.
Also, it is readily seen that $u_i \in \cC(\overline H_T)$ ($i=1,2$).
Therefore, it is clear that $u=u_1+u_2$ satisfies the boundary condition
\eqref{eq:I06a}.

Let us further analyze the function $u_2$.
Once we prove that $u_2$ is also a classical solution of \eqref{eq:I05e}
in $H_T$, then we are done.
Let $\{g_k\}_{k=1}^\infty$ be be smooth approximations of $f_2$, say obtained
by using mollifiers, such that $g_k\to f_2$ uniformly.
Denote
\begin{equation*}
v_k(t,x)= \bE g_k(X_T(t,x)).
\end{equation*}
Then by the same reasoning as above, the functions $\{v_k\}_{k=1}^\infty$
are classical solution of \eqref{eq:I05e} in $H_T$.
Note that by interior Schauder estimates, $\Ho$-norm of $v_k$ in any compact set
belonging to $H_T\setminus\{(t,x) : x=b(t)\}$ is estimated through
its supremum over a bounded domain containing the set.
Since $g_k\to f_2$ uniformly, we also have $v_k \to u_2$ uniformly, and thus
we get
\begin{equation*}
u_2\in \Ho_{loc}(\Omega),\quad\text{where }\,
\Omega:=H_T\setminus \{(t,x) : x=b(t)\}
\end{equation*}
and satisfies the equation \eqref{eq:I05e} in $\Omega$.

Next, we claim that $u_2\equiv 0$ in $\{(t,x)\in [0,T]\times\bR :x\geq b(t)\}$.
Note that the process
\begin{eqnarray*}
Y_{s}(t,x):=X_{s}(t,x)-b_{s}\qquad (b_s=b(s))
\end{eqnarray*}
satisfies the following stochastic differential equation:
\begin{equation}
\label{eq:P05f}
\left\{\begin{array}{l}
dY_{s}(t,x)=-Y_{s}(t,x)\,dw_{s}-b'(s)\,ds,\quad s\ge t,\\
Y_{t}(t,x)=x-b(t).
\end{array}\right.
\end{equation}
The solution to \eqref{eq:P05f} is unique and has a representation
\begin{equation*}
Y_s=Y_t\,e^{w_t-w_s+\frac{1}{2}(t-s)}-\int_t^s e^{w_r-w_s+\frac{1}{2}(r-s)}\,
b'(r)\,dr,\quad s\geq t.
\end{equation*}
Therefore, from the assumption $b'\leq 0$, we conclude that $Y_{s}(t,x)\geq 0$
for all $s\geq t$ provided that $Y_{t}(t,x)=x-b(t) \geq 0$.
In particular, we have $X_T(t,x)=X_T(t,x)-b(T)=Y_T(t,x)\ge 0$ if $x\geq b(t)$.
Therefore, from \eqref{eq:P01a} and the fact that
$f_2\equiv 0$ for $x\ge 0$, we find $u_2(x,t)=0$ if $x\ge b(t)$.
We have thus proved the claim that 
$u_2\equiv 0$ in $\{(t,x)\in [0,T]\times\bR :x\geq b(t)\}$.

Now, we will show that $u_2 \in \Co_{loc}(H_T)$.
To comply with standard conventions in parabolic PDE theory, we make a
change of variable $t\mapsto T-t$ and denote
\begin{equation}
\label{eq:P04z}
v(t,x):=u_2(T-t,x)\quad\text{and}\quad \psi(t):=b(T-t).
\end{equation}
By the observations made above, we have
\begin{equation}
\label{eq:P04k}
v\in \cC(\overline H_T)\cap \Ho_{loc}(H_T\setminus\Gamma),\quad\text{where }\,
\Gamma:=\{(t,x)\in H_T : x=\psi(t)\},
\end{equation}
and satisfies the equation
\begin{equation*}
v_t- \tfrac{1}{2}(x-\psi(t))^2 v_{xx}=0 \quad\text{in}\quad
H_T\setminus \Gamma.
\end{equation*}
In order to show that $v\in \Co_{loc}(H_T)$, we need investigate the behavior
of $v$ near $\Gamma$.
By \eqref{eq:I07b}, we find that $\phi:=\psi^{-1}$ is defined on $[0,\ell]$,
where $\ell:=\psi(T)$, and satisfies 
\begin{equation*}
1/m_2\le \phi'(x) \le 1/m_1,\,\text{ for a.e. }\, x\in (0,\ell).
\end{equation*}
In the rest of the proof, we use the following notation.
For $z_0=(t_0,x_0)\in\bR^2$, we denote
\begin{align*}
C_r(z_0)&=\{(t,x)\in \bR^2:|t-t_0|<r,\,|x-x_0|<(m_1/2)r\},\\
\U_r(z_0)&= C_r(z_0) \cap \{(t,x)\in H_T: x<\psi(t)\},\\
\U_r'(z_0)&= \{(t,x)\in \U_r(z_0): |x-x_0|<(m_1/4)r\},\\
\Gamma_r(z_0)&= C_r(z_0)\cap \Gamma.
\end{align*}

\begin{lemma}[Key lemma]
\label{lem:P01}
Let $z_0=(t_0,x_0)=(t_0,\psi(t_0))\in \Gamma$ and $r\in (0,1)$
be any number satisfying $\overline C_r(z_0)\subset D:=(0,T)\times (0,\ell)$.
Then, the function $v$ defined as in \eqref{eq:P04z} satisfies
\begin{equation}
\label{eq:P12t}
|v|_{0;\U_r'(z_0)}\leq N_0 r^{1/2} e^{-k_0/r} |v|_{0;D},
\end{equation}
where $N_0=N_0(m_1,m_2)$ and $k_0=k_0(m_1,m_2)>0$.
Moreover, we have
\begin{equation}
\label{eq:P12q}
r^{3/2}|v_x(t_0+r,x_0)|+ r^3|v_{xx}(t_0+r,x_0)|+ r |v_t(t_0+r,x_0)|
\leq N_1 r^{1/2} e^{-k_0/r} |v|_{0;D},
\end{equation}
where $N_1=N_1(m_1,m_2)$.
\end{lemma}

\begin{proof}
Let $T$ be a linear mapping defined by 
\begin{equation}
\label{eq:P12z}
T(t,x):=\left((t-t_0)/r,(x-x_0)/cr^{3/2}\right),
\quad \text{where }\,c:=(m_1+2m_2)/\sqrt{8}.
\end{equation}
We shall denote $\Omega_r:=T\left(\U_r(z_0)\right)$,
$\Sigma_r:=T\left(\Gamma_r(z_0)\right)$, and
\begin{equation*}
Q_r:=T\left(C_r(z_0)\right)=\{(t,x)\in\bR^2:|t|<1,\,|x|<(m_1/2c)r^{-1/2}\}.
\end{equation*}
We also define the functions $w(t,x)$ and $a(t,x)$ on $\overline Q_r$ by
\begin{align}
\label{eq:P13a}
w(t,x)&:=v\circ T^{-1}(t,x)=v(t_0+r t,x_0+cr^{3/2} x),\\
\label{eq:P13b}
a(t,x)&:= \frac{1}{2(cr)^2}\left(x_0+cr^{3/2} x-\psi(t_0+rt)\right)^2.
\end{align}
Then $w \in \Ho_{loc}(\Omega_r)\cap \cC(\overline\Omega_r)$ and satisfies
\begin{equation}
\label{eq:P13s}
\left\{\begin{array}{l}
L w:=w_t-a(t,x)w_{xx}=0\quad\text{in}\quad \Omega_r,\\
w=0\quad\text{on}\quad \Sigma_r,
\end{array}\right.
\end{equation}
Note that $a(t,x)$ satisfies the following inequalities in $Q_r$.
\begin{align}
\label{eq:P19v}
0\leq a(t,x)
&=\frac{1}{2(cr)^2}\left(x_0+cr^{3/2}x-\psi(t_0+rt)+\psi(t_0)-x_0\right)^2\\
\nonumber
&\leq \frac{1}{2(cr)^2}\left(cr^{3/2}|x|+m_2r|t|\,\right)^2
\leq\frac{1}{8c^2}(m_1+2m_2)^2 =1.
\end{align}
Also, observe that $\Sigma_r\subset \{(t,x)\in \bR^2: |t|<1/2\}$.
By \eqref{eq:P13s} and \eqref{eq:P19v}, we may apply Lemma~\ref{lem:G01}
to $u(t,x)=w(t+1,x)$ with $R=(m_1/2c)r^{-1/2}$ to conclude that
\begin{equation}
\label{eq:P14w}
|w|_{0;\Omega_r'} \leq
N r^{1/2}e^{-k_0/r}\,|w|_{0;\Omega_r},
\end{equation}
where $\Omega'_r=T(\U_r'(z_0))$, $N_0=8(m_1+2m_2)/\sqrt{\pi}m_1$,
and $k_0=m_1^2/16(m_1+2m_2)^2$.
It is obvious by \eqref{eq:P13a} that \eqref{eq:P12t} follows from \eqref{eq:P14w}.

Next, we turn to the proof of \eqref{eq:P12q}.
Note that by a similar calculation as in \eqref{eq:P19v}, we have
(recall $0<r<1$)
\begin{equation}
\label{eq:P17a}
\norm{\partial_x a}_{L^\infty(Q_r)}\leq 4(m_1+2m_2),\quad
\norm{\partial_t a}_{L^\infty(Q_r)}\leq 4m_2/(m_1+2m_2).
\end{equation}
Let us denote $\Pi_\rho:=(1-\rho^2,1)\times(-\rho,\rho)$ for $\rho>0$.
Note that if $(t,x)\in \Pi_\rho$, then
\begin{align}
\label{eq:P19w}
a(t,x)
&\geq \frac{1}{2(cr)^2}
\left(|\psi(t_0)-\psi(t_0+r)|-|\psi(t_0+r)-\psi(t_0+rt)|-cr^{3/2}|x|\right)^2\\
\nonumber
&\geq \frac{1}{2(cr)^2} \left(m_1 r-m_2r\rho^2-cr^{3/2}\rho\right)^2
\geq \frac{1}{2c^2} \left(m_1-m_2\rho^2-c\rho\right)^2.
\end{align}
Fix $\rho_0=\rho_0(m_1,m_2) \in (0,1/2]$ such that
\begin{equation*}
m_1-m_2\rho_0^2-c\rho_0 \geq m_1/2\quad\text{and}\quad
\Pi_{\rho_0}\subset \Omega_r'.
\end{equation*}
Then by \eqref{eq:P19v} and \eqref{eq:P19w}, we have
\begin{equation}
\label{eq:P18b}
2m_1/(m_1+2m_2)^2  \leq a(t,x) \leq 1,\quad \forall (t,x)\in \Pi_{\rho_0}.
\end{equation}
By \eqref{eq:P17a}, \eqref{eq:P18b}, and the interior Schauder estimates,
we have
\begin{equation}
\label{eq:P19c}
|w_x(1,0)|+|w_{xx}(1,0)|+|w_t(1,0)|\leq C
|w|_{0;\Pi_{\rho_0}},
\end{equation}
where $C=C(m_1,m_2)$; see e.g. \cite{Lieberman}.
Now, the estimate \eqref{eq:P12q} follows from \eqref{eq:P13a}, \eqref{eq:P14w},
and \eqref{eq:P19c}.
The lemma is proved.
\end{proof}

We are ready to prove that $v\in \Co_{loc}(H_T)$.
We define $v_x=0$ (resp. $v_{xx}=0$, $v_t=0$) on $\Gamma$.
By \eqref{eq:P04k}, it is enough to show that $v_x$ (resp. $v_{xx}, v_t$) is
continuous at each $z_0=(t_0,x_0)\in\Gamma$.
Fix an $r_0=r_0(z_0)\in(0,1)$ such that
$\overline C_{r_0}(z_0)\subset D=(0,T)\times (0,\ell)$.
Note that for any $z_1\in \Gamma_{r_0/4}(z_0)$ and $r<r_0/4$, we have
$C_r(z_1)\subset C_{r_0}(z_0)$.
Therefore, by Lemma~\ref{lem:P01}
\begin{equation}
\label{eq:P29a}
|w(\phi(x)+r,x)| \leq N_1 r^{-\beta} e^{-k_0/r} |v|_{0;D},
\quad \forall r \in (0,r_0/4)\quad\forall x \in (x_0-r_0/4,x_0+r_0/4),
\end{equation}
where $w:=v_x$ (resp. $w:=v_{xx}$, $w:=v_t$) and $\beta=-1$
(resp. $\beta=-5/2$, $\beta=-1/2$).
On the other hand, note that there is some $\delta=\delta(m_1,m_2)>0$ such that
\begin{equation}
\label{eq:P29b}
\U_{\delta r_0}(z_0)\subset
\{(\phi(x)+r,x)\in \bR^2: 0<r<r_0/4, |x-x_0|<r_0/4\}.
\end{equation}
From \eqref{eq:P29a} and \eqref{eq:P29b}, we find that
$\lim_{\rho\to 0}|w|_{0;C_\rho(z_0)}=0$. 
The theorem is proved.
\section{Generalization of Key lemma}\label{sec:extra}
Let $\phi:\bR^n\to\bR$ be a Lipschitz continuous function satisfying
$\|\nabla \phi\|_{L^\infty(\bR^n)}\leq M_0$ for some $M_0 \in (0,\infty)$ and
denote
\begin{equation*}
\Gamma:=\set{(t,x)\in \bR\times \bR^n : t=\phi(x)}.
\end{equation*}
For $z=(t,x)\in \bR\times\bR^n$ and $r>0$, we shall denote
\begin{align*}
C_r(z)&:=\{(s,y)\in \bR\times\bR^n: |s-t|<r, \,
\max_{1\leq k\leq n}|y_k-x_k|<(1/2M_0)r \},\\
\U_r(z)&:=C_r(z) \cap \{(s,y)\in \bR\times \bR^n : s>\phi(y)\},\\
\U_r'(z)&:=\{(s,y)\in \U_r(z): \max_{1\leq k \leq n} |y_k-x_k|<(1/4M_0)r\},\\
\Gamma_r(z)&:=C_r(z)\cap\Gamma.
\end{align*}

\begin{theorem}
\label{thm:E01}
Let $z_0\in \Gamma$ and $r>0$ be given.
Assume that there are numbers $\mu>1$ and $\Lambda>0$ such that
the coefficients $\big(a_{ij}(t,x)\big)_{i,j=1}^n$ satisfy
\begin{equation}
\label{eq:E01a}
0\leq a_{ij}(t,x)\xi_i\xi_j\leq \Lambda\abs{\phi(x)-t}^\mu\abs{\xi}^2,\quad
\forall (t,x)\in C_r(z_0),\quad \forall \xi\in\bR^n.
\end{equation}
Let $u\in\Co_{loc}\big(\U_r(z_0)\big)\cap \cC\big(\overline \U_r(z_0)\big)$
satisfy
\begin{equation*}
\left\{\begin{array}{l}
Lu:= u_t- a_{ij}D_{ij}u=0\quad\text{in}\quad \U_r (z_0),\\
u=0 \quad\text{on}\quad \Gamma_r(z_0).
\end{array}\right.
\end{equation*}
Then the following estimate holds.
\begin{equation}
\label{eq:E02b}
|u|_{0;\U_r'(z_0)}\le N_0 r^{(\mu-1)/2} e^{-k_0 r^{1-\mu}}\, |u|_{0;\U_r(z_0)},
\end{equation}
where $N_0=N_0(n,\mu,\Lambda,M_0)$ and $k_0=k_0(\mu,\Lambda,M_0)>0$.
\end{theorem}
\begin{proof}
The proof is a slight modification of that in Lemma~\ref{lem:G01}.
By renormalizing $u$ to $u/\abs{u}_{0;\U_r(z_0)}$,
we may assume $\abs{u}_{0;\U_r(z_0)}=1$.
Let $T$ be a linear mapping defined by
\begin{equation}
\label{eq:E02h}
T(t,x):=\left((t-t_0)/r,(x-x_0)/cr^{(1+\mu)/2}\right),
\quad \text{where }\,c:=\Lambda^{1/2} (3/2)^{\mu/2}.
\end{equation}
Denote $\Omega_r:=T\left(\U_r(z_0)\right)$,
$\Omega_r':=T\left(\U_r'(z_0)\right)$,
$\Sigma_r:=T\left(\Gamma_r(z_0)\right)$, and
\begin{equation}
\label{eq:E03f}
Q_r:=T\left(C_r(z_0)\right) =\{(t,x)\in\bR\times \bR^n:|t|<1,\,
\max_{1\leq k \leq n}|x_k|<(1/2cM_0)r^{(1-\mu)/2}\}.
\end{equation}
Define the functions $w(t,x)$ and $\tilde a_{ij}(t,x)$ on $\overline \Omega_r$
and $Q_r$, respectively, by
\begin{align}
\label{eq:E13a}
w(t,x)&:=u\circ T^{-1}(t,x)=u(t_0+r t,x_0+cr^{(1+\mu)/2} x),\\
\label{eq:E13b}
\tilde a_{ij}(t,x)&:= (c^2 r^\mu)^{-1}\,a_{ij}(t_0+r t,x_0+cr^{(1+\mu)/2} x).
\end{align}
Then $w \in \Co_{loc}(\Omega_r)\cap \cC(\overline\Omega_r)$ and satisfies
\begin{equation}
\label{eq:E13s}
\left\{\begin{array}{l}
\tilde L w:=w_t-\tilde a_{ij}(t,x)D_{ij}w=0\quad\text{in}\quad \Omega_r,\\
w=0\quad\text{on}\quad \Sigma_r,
\end{array}\right.
\end{equation}
By \eqref{eq:E01a} and \eqref{eq:E13a}, for all $(t,x)\in Q_r$
and $\xi\in \bR^n$, we have
\begin{align}
\label{eq:E19v}
0\leq \tilde a_{ij}(t,x) \xi_i \xi_j
&\leq \frac{\Lambda}{c^2 r^\mu}\left(|\phi(x_0+cr^{(1+\mu)/2}x)-\phi(x_0)|
+r|t|\right)^\mu |\xi|^2\\
\nonumber
&\leq \frac{\Lambda}{c^2 r^\mu}
\left(M_0 cr^{(1+\mu)/2}|x|+r|t|\,\right)^\mu |\xi|^2 \leq
\frac{\Lambda}{c^2}(3/2)^\mu |\xi|^2 =|\xi|^2.
\end{align}
Let $v$ be given as in \eqref{eq:G05i} with $R=(1/2cM_0)r^{(1-\mu)/2}$
and define
\begin{equation}
\label{eq:E20t}
V(t,x)=V(t,x_1,\ldots,x_n):= \sum_{k=1}^n v(t+1,x_k).
\end{equation}
Then, since $v_{xx} \geq 0$ by \eqref{eq:G06a} and $\tilde a_{kk}\leq 1$,
for all $k=1,\ldots,n$, by \eqref{eq:E19v}, we have
\begin{align*}
\tilde L V &= V_t- \tilde a_{ij} D_{ij} V=
\sum_{k=1}^n\left( v_t(t,x_k)- \tilde a_{kk} v_{xx}(t,x_k)\right) \\
\nonumber
&\geq \sum_{k=1}^n\left( v_t(t,x_k)- v_{xx}(t,x_k)\right)=0
\quad\text{in}\quad Q_r.
\end{align*}
Note that by \eqref{eq:G05j}, $V \geq 1$ on
$\partial_x Q_r:=\{(t,x)\in \bR\times\bR^n: |t|<1,\, |x_k|=R,\,
\forall k=1,\ldots,n\}$.
Also, observe that $\Sigma_r\subset \{(t,x)\in Q_r: |t|<1/2\}$.
Therefore, we have $V \geq |w|$ on the parabolic boundary $\partial_p\Omega_r$
of $\Omega_r$.
Then, by the comparison principle, we obtain
\begin{equation}
\label{eq:E21a}
|w(t,x)|\leq V(t,x), \quad\forall (t,x)\in \Omega_r.
\end{equation}
On the other hand, by \eqref{eq:G06a}, \eqref{eq:G07b}, and \eqref{eq:E20t},
we have (recall $R=(1/2cM_0)r^{(1-\mu)/2}$)
\begin{equation}
\label{eq:E22b}
V(t,x)\leq (32ncM_0/\sqrt{2\pi}) r^{(\mu-1)/2} e^{-r^{1-\mu}/128c^2M_0^2}
\quad\forall(t,x) \in \Omega_r'.
\end{equation}
We obtain \eqref{eq:E02b} by combining \eqref{eq:E13a}, \eqref{eq:E21a},
and \eqref{eq:E22b}. The theorem is proved.
\end{proof}

\begin{theorem}
\label{thm:E02}
Let $\bar z_0\in \Gamma$ and $R>0$ be given.
Assume that there are numbers $\mu>1$ and $\lambda, \Lambda, M_1>0$ such that
the coefficients $\big(a_{ij}(t,x)\big)_{i,j=1}^n$ satisfy
\begin{gather}
\label{eq:E31a}
\lambda\abs{\phi(x)-t}^\mu\abs{\xi}^2 \leq
a_{ij}(t,x)\xi_i\xi_j\leq \Lambda\abs{\phi(x)-t}^\mu\abs{\xi}^2,\quad
\forall (t,x)\in C_R(\bar z_0),\quad \forall \xi\in\bR^n,\\
\label{eq:E31b}
|\nabla_{t,x} a_{ij}(t,x)|\leq M_1 \abs{\phi(x)-t}^{\mu-1},\quad
\text{for a.e. }\, (t,x)\in C_R(\bar z_0).
\end{gather}
Suppose $u\in\Ho_{loc}\big(\U_R(\bar z_0)\big)\cap \cC\big(\overline \U_R(\bar z_0)\big)$,
for some $\alpha\in(0,1)$, and satisfies
\begin{equation*}
\left\{\begin{array}{l}
Lu:= u_t- a_{ij}D_{ij}u=0\quad\text{in}\quad \U_R (\bar z_0),\\
u=0 \quad\text{on}\quad \Gamma_R(\bar z_0).
\end{array}\right.
\end{equation*}
Then if we extend $u\equiv 0$ in $C_R(\bar z_0)\setminus \U_R(\bar z_0)$,
we have $u\in \Co_{loc}(C_{R/2}(\bar z_0))$
\end{theorem}
\begin{proof}
Let $z_0=(t_0,x_0)=(\phi(x_0),x_0) \in \Gamma_{R/2}(\bar z_0)$ and
let $0<r<\min(1,R/2)$ so that $r<1$ and $C_r(z_0)\subset C_R(\bar z_0)$.
Then, by \eqref{eq:E02b} of Theorem~\ref{thm:E01} we find
\begin{equation}
\label{eq:E32b}
|u|_{0;\U_r'(z_0)}\le N_0 r^{(\mu-1)/2} e^{-k_0 r^{1-\mu}}\, |u|_{0;\U_R(\bar z_0)},
\end{equation}
Let $T$, $Q_r$, $w(t,x)$, and $\tilde a_{ij}(t,x)$ be defined
as in \eqref{eq:E02h} -- \eqref{eq:E13b}.
Then, by \eqref{eq:E31b} we have
\begin{equation}
\label{eq:E22y}
\|\nabla_{t,x} \tilde a_{ij}\|_{L^\infty(Q_r)}\leq C M_1,
\quad\text{where }\,C=C(\Lambda,\mu).
\end{equation}
Denote $\Pi_\rho:=(1-\rho^2,1)\times (-\rho,\rho)^n$.
Note that if $(t,x)\in \Pi_\rho$, then we have 
\begin{equation*}
|\phi(x_0+cr^{(1+\mu)/2}x)-(t_0+rt)|\geq
|r(1-\rho^2)-M_0 cr^{(1+\mu)/2}\rho| \geq r|1-\rho^2-M_0c\rho|.
\end{equation*}
Let us fix a number $\rho_0=\rho_0(\mu,\Lambda,M_0) \in (0,1/2)$ such that
$|1-\rho^2-M_0c\rho|\geq (1/2)^{1/\mu}$ and $\Pi_{\rho_0}\subset \Omega'_r$.
Then, it follows from \eqref{eq:E13b} and \eqref{eq:E31a} that
\begin{equation}
\label{eq:E22x}
\tilde a_{ij}(t,x) \xi_i \xi_j \geq (\lambda/2c^2) |\xi|^2 = (\lambda/\Lambda)
2^{\mu-1}3^{-\mu}|\xi|^2.
\end{equation}
Then by \eqref{eq:E22y}, \eqref{eq:E19v}, \eqref{eq:E22x}, and the interior
Schauder estimate, we have
\begin{equation}
\label{eq:E39c}
|D_x w(1,0)|+|D_x^2 w(1,0)|+|w_t(1,0)|\leq C
|w|_{0;\Pi_{\rho_0}},
\end{equation}
where $C=C(n,\alpha,\mu,\lambda,\Lambda,M_0,M_1)$.
Therefore, by using \eqref{eq:E13a} and \eqref{eq:E32b}, we conclude
\begin{align}
\label{eq:E42q}
r^{(1+\mu)/2}|D_x u(t_0+r,x_0)| &+ r^{1+\mu}|D_x^2 u(t_0+r,x_0)|+ r |u_t(t_0+r,x_0)|\\
\nonumber
&\leq N_1 r^{(\mu-1)/2} e^{-k_0 r^{1-\mu}} |u|_{0;\U_R(\bar z_0)},
\end{align}
where $N_1=N_1(n,\alpha,\mu,\lambda,\Lambda,M_0,M_1)$.
Finally, by using \eqref{eq:E42q} instead of \eqref{eq:P12q} and proceeding
similarly as in the proof of Theorem~\ref{thm:I01}, we see that
$u\in \Co_{loc}(C_{R/2}(\bar z_0))$.
This completes the proof.
\end{proof}

\noindent
\textbf{Acknowledgments} The author is grateful to Mikhail Safonov and Jan Ve\v{c}e\v{r} for very helpful discussion.


\end{document}